\documentclass[a4paper]{article}
\usepackage{amsthm,amssymb,amsmath,enumerate}

\usepackage{booktabs}

\usepackage{hyperref}
\hypersetup{
	colorlinks=true,  % instead of frames around the links, colour links
	linkcolor=black,    % colour internal links
	citecolor=black,    % colour citation links
	urlcolor=black      % colour of external links
}

\usepackage{tikz}
\usetikzlibrary{calc}
\usetikzlibrary{backgrounds}
\tikzstyle{hvertex}=[thick,circle,inner sep=0.cm, minimum size=2mm, fill=white, draw=black]
\tikzstyle{hedge}=[very thick]
\tikzstyle{harrow}=[thick,arrows=->]
\tikzstyle{darrow}=[thick,arrows=<-]
\tikzstyle{rededge}=[very thick,red]
\tikzstyle{point}=[draw,circle,inner sep=0.cm, minimum size=1mm, fill=black]
\tikzstyle{pointer}=[thick,->,shorten >=2pt,color=grau]
\tikzstyle{facebdry}=[color=auchblau, very thick] % face boundary
\tikzstyle{face}=[facebdry,fill=hellblau]
\tikzstyle{nface}=[color=hellblau,fill=hellblau,thick] % naked face, without boundary
\tikzset{>={latex}}
\tikzstyle{tinyvx}=[thick,circle,inner sep=0.cm, minimum size=1.3mm, fill=white, draw=black]

\colorlet{auchblau}{blue!60!white}
\colorlet{hellblau}{blue!20!white}
\colorlet{hellrot}{red!40!white}
\colorlet{grau}{black!50!white}

\newtheorem{definition}{Definition}
\newtheorem{proposition}[definition]{Proposition}
\newtheorem{theorem}[definition]{Theorem}

\newtheorem{lemma}[definition]{Lemma}

\newtheorem{problem}[definition]{Problem}

\newcommand{\comment}[1]{}
\newcommand{\N}{\mathbb N}

\newcommand{\polylog}{{\rm polylog}}

\newcommand{\sm}{\setminus}

\newcommand{\new}[1]{#1}
\newenvironment{newe}{}{}
\newcommand{\bn}{\begin{newe}}
\newcommand{\en}{\end{newe}}

\title{Frames, $A$-paths and the Erd\H os-P\'osa property}
\author{Henning Bruhn\thanks{Partially supported by DFG, grant no.\  BR 5449/1-1} \and Matthias Heinlein \and Felix Joos\thanks{The research was also supported by the EPSRC, grant no.\ EP/M009408/1.}}
\date{}

\begin{document}

\maketitle

\begin{abstract}
A key feature of
Simonovits' proof of the classic Erd\H os-P\'osa theorem is a simple subgraph of the host graph, a \emph{frame},
that determines the outcome of the theorem. We transfer this frame technique to $A$-paths. With it 
we deduce a simple proof of Gallai's theorem, although with a worse bound, and we verify the Erd\H os-P\'osa 
property for long and for even $A$-paths. 
We also show that even $A$-paths do not have the edge-Erd\H os-P\'osa property. 
\end{abstract}

\section{Introduction}

If a graph is far away from being a tree, is there a simple certificate for that --- for instance
a large number of \new{vertex-}disjoint cycles? The classic theorem of Erd\H os and P\'osa asserts that, 
yes, there is such a certificate:

\begin{theorem}[Erd\H os and P\'osa~\cite{EP65}]\label{epthm}
For every positive integer $k$, 
every graph either contains $k$ \new{vertex-}disjoint cycles
or a set of $O(k\log k)$ vertices that meets every cycle.
\end{theorem}

The proof of Erd\H os and P\'osa, while not long, is somewhat indirect and in particular
rests on a result about $A$-paths of Gallai. (We will come back to that.) 
Simonovits~\cite{Sim67} gave a different and cleaner proof that works in two steps. 
First Simonovits showed that a large  cubic multigraph always contains many disjoint cycles.
\new{(Whenever we write ``disjoint'' we mean ``vertex-disjoint''.)}
\begin{lemma}[Simonovits~\cite{Sim67}]\label{framelem}
Every cubic multigraph with at least $(4+o(1))k\log k$ vertices contains $k$ disjoint cycles.\sloppy 
\end{lemma}
In the second step, Simonovits considered a maximal subgraph $F$ of the graph $G$ with all
degrees between $2$ and $3$. Then, if $F$ has many vertices of degree~$3$, 
the lemma yields $k$ disjoint cycles, and if not then the vertices of degree~$3$ will (almost)
be a \emph{hitting set}; that is, a set meeting all cycles.\footnote{We gloss over some technicalities here that, however, are not 
hard to resolve. In particular, components of $F$ that consists of single cycles need to be taken
care of, as well as cycles that meet $F$ in a single vertex.}

We think of the graph $F$ in Simonovits' argument as a \emph{frame}, a subgraph that
essentially  captures 
 all target objects, the cycles, but has a much simpler structure. 
In particular, the frame alone determines whether we find $k$ disjoint cycles or a hitting set, 
both of which can be (essentially) obtained directly from the frame. 

\new{
Pontecorvi and Wollan~\cite{PW12} extended the frame argument  
to capture also \emph{$A$-cycles}, cycles that each contain at least one vertex of a fixed set $A$.
Fiorini and Herinckx~\cite{FH14} used the technique in a different way to treat \emph{long cycles}.  
Bruhn, Joos and Schaudt~\cite{BJS14} further modified the technique so that 
it also applies to long $A$-cycles.
Finally, Mousset, Noever, \v{S}kori\'c, and Weissenberger~\cite{MNSW16} again refined the 
frame argument in order to achieve
an asymptotically tight bound (up to a constant factor) on the size of the hitting set for long cycles.
}
A frame also plays an important role in our article~\cite{FH14} on edge-disjoint long cycles; see Section~\ref{edgesec}.

We demonstrate here that a simple frame argument also works for Erd\H os-P\'osa type theorems 
about different variants of $A$-paths. 
In this case, trees with their leaves in $A$ will play the role of a frame.
We see the main merit of this article in discussing the frame argument and 
in showing how it can be used in a different context. 
Along the way, we will obtain several \new{new} results.

\bigskip

Gallai discovered that \emph{$A$-paths}, paths with first and last vertex but no interior vertex in $A$,
behave in a quite similar way as the cycles in Erd\H os and P\'osa's theorem.
\begin{theorem}[Gallai~\cite{Gal61}]\label{gallaithm}
For every positive integer $k$, every graph $G$ and 
every set $A\subseteq V(G)$,
 the graph $G$ either contains $k$ disjoint  
$A$-paths or a vertex set $X$ of size $|X|\leq 2k-2$ that meets every  $A$-path.
\end{theorem}

A class of objects, such as $A$-paths or cycles,
has the \emph{Erd\H os-P\'osa property} if they satisfy a theorem similar to Theorems~\ref{epthm} and~\ref{gallaithm}:
namely that in every graph there are either $k$ disjoint such objects or that there is a set of vertices
that meets each target object and whose size is bounded by function that depends only on $k$.\footnote{That the notion 
is somewhat vague is on purpose. We want it to cover types of subgraphs, such as cycles, even cycles and so on, 
as well as subgraphs with additional structure, such as $A$-paths. To squeeze all these into one formally correct
notion seems too much effort for the little benefit.}  

Thus, cycles and $A$-paths have the {Erd\H os-P\'osa property}, but also many other graph classes
such as even cycles~\cite{Tho88}, graphs that contract to a fixed planar graph~\cite{RS86}, or different variants of $A$-paths~\cite{CGGGLS06,GGRSV09,KKM11,Wol10}.

\label{labelledGraphs}
Very general types of $A$-paths can be realised by \new{labeling} the edges of the graph 
with \new{elements} from an (abelian) group~$\Gamma$.
There are at least \new{two ways to do that}. 
\new{
In the simpler way,  \emph{undirected group labelings}, 
every edge $e$ receives a label $\gamma(e)\in\Gamma$. 
A path $P$ is then said to be \emph{non-zero} if the sum of its edge labels is non-zero in $\Gamma$.
If, for instance, the group $\Gamma$ is $\mathbb Z_2$ and every edge receives a label of~$1$,
then the non-zero paths are simply the paths of odd length.
With respect to undirected labelings, 
Wollan~\cite{Wol10} proved:
}

\begin{theorem}[Wollan~\cite{Wol10}]\label{wolthm}
For every graph $G$, vertex set $A\subseteq V(G)$, 
abelian group $\Gamma$ and undirected $\Gamma$-edge labeling 
of $G$, there are either $k$ disjoint non-zero $A$-paths or 
a vertex set of size $O(k^4)$ vertices that meets every non-zero $A$-path.
\end{theorem}
\new{
Chudnovsky et al.~\cite{CGGGLS06} investigated \emph{directed group labelings}.
In contrast to undirected labelings, an edge $e$ of a path is 
counted with weight $\gamma(e)$ if it is traversed 
in direction of a fixed reference orientation, and with weight $-\gamma(e)$ otherwise. 
Chudnovsky et al.\ prove a result similar to Theorem~\ref{wolthm} but with 
a much smaller hitting set, namely one of size at most~$2k-2$,
which is clearly optimal. 
}

Both results have \new{interesting} consequences: 
odd $A$-paths have the Erd\H os-P\'osa property (Geelen et al.~\cite{GGRSV09}),
 \new{and so do} \emph{$A$--$B$--$A$-paths},
\new{i.e.\ $A$-paths}  that each meet some vertex from a second vertex set~$B$
(see also Kakimura, Kawarabayashi, and Marx~\cite{KKM11}).

Other extensions of Gallai's theorem concern directed $A$-paths (Kriesell~\cite{Kri05})
and edge-disjoint $A$-paths (Mader~\cite{Mad78}). %We discuss edge-disjoint $A$-paths in Section~\ref{edgesec}.

In this article, 
we \new{apply} the frame argument to show that $A$-paths (Section~\ref{sec:Gallai}), 
long $A$-paths (Section~\ref{longsec}), 
even $A$-paths (Section~\ref{evensec}) as well as 
certain trees with their leaves in $A$
(Section~\ref{sec:Atrees}) have the Erd\H os-P\'osa
property.
We also discuss types of $A$-paths that do not have the Erd\H os-P\'osa property.
These are $A$-paths with certain, more complicated, modularity constraints (Section~\ref{evensec})
as well as directed even $A$-paths or directed $A$--$B$--$A$-paths (Section~\ref{secdirect}).
In Section~\ref{edgesec} we turn to the edge-version of the Erd\H os-P\'osa property. 
In particular, we give a simple proof of 
 Mader's theorem about edge-disjoint $A$-paths (but with a worse bound),
and we show that neither even $A$-paths nor $A$--$B$--$A$-paths have the edge-Erd\H os-P\'osa
property. To the best of our knowledge this is the first example of a class 
that possesses the vertex-Erd\H os-P\'osa property but not the edge-Erd\H os-P\'osa property.

\section{Gallai's theorem}\label{sec:Gallai}

While Gallai's original proof of his theorem was quite technical, there are very nice proofs
that reduce the problem to matchings and in particular to the Tutte-Berge formula; see Schrijver~\cite{Sch01}.

We give an alternative proof of Gallai's theorem, albeit with a slightly worse \new{bound on the } size of the hitting 
set. Gallai's bound of $2k-2$, on the other hand,  is optimal as can be seen by considering
a complete graph on $2k-1$ vertices, all in $A$. 

Still, because the proof is quite simple and because it serves as a model for the other results later, 
we find it worth the effort. 
First, we set up the frame, the structure that captures the essence of the $A$-paths in the graph. 
The frame could not be much simpler: it is a tree with all its leaves in $A$. 
The following easy lemma plays the same role as Lemma~\ref{framelem} in Simonovits' proof.

\begin{lemma}\label{l2lpaths}
Every subcubic tree with $p$ leaves contains $\lfloor\tfrac{p}{2}\rfloor$ disjoint
leaf-to-leaf paths.
\end{lemma}
\begin{proof}
Let $T$ be a subcubic tree with $p$ leaves.
We \new{proceed by} induction on the number of leaves. 
 By contracting the edges incident with vertices of degree~$2$,
we obtain a new tree whose leaf-to-leaf~paths are in direct correspondence to 
leaf-to-leaf~paths in $T$. We thus may assume that $T$ contains no vertices of degree~$2$.

If $T$ has at most three leaves, the statement is obviously true. Assume $T$ to have at least four leaves,
and pick  a root~$r$. Choose a vertex $t$ of degree~$3$ that is farthest from $r$. 
Then $t$ is adjacent to two leaves~$\ell_1$ and $\ell_2$ and one non-leaf $s$ (as $T$ has at least four leaves). 
We remove the path $P=\ell_1t\ell_2$ from $T$ 
and obtain a new tree $T'$. 
As $s$ has degree $3$ in $T$, this implies that the sets of leaves of $T$ and of $T'$ 
differ only in $\ell_1$ and $\ell_2$. 
Inductively, we  obtain $\lfloor\tfrac{p-2}{2}\rfloor$ disjoint leaf-to-leaf
paths in $T'$ and thus in $T$. Together with $P$ we find the desired paths.
\end{proof}

We prove Gallai's theorem with a hitting set of size at most~$4k$.

\begin{proof}[Proof of Theorem~\ref{gallaithm} with hitting set of size at most $4k$]
Let \new{$F\subseteq G$} be a forest maximal under inclusion such that
\begin{itemize}
\item $F$ \new{is subcubic with no isolated vertices}; and %has minimum degree~$1$ and maximum degree at most~$3$; and
\item every leaf of $F$ lies in $A$, and every vertex in $A\cap V(F)$ is a leaf of $F$.
\end{itemize}
Let $c$ be the number of components of $F$.
If $F$ contains $2k+c$ vertices of $A$, then by applying Lemma~\ref{l2lpaths} to each component
of $F$ we obtain $k$ disjoint leaf-to-leaf paths in $F$, each of which is an $A$-path.
\comment{Indeed, if $F$ has $k$ or more components then, choosing one $A$-path from every component, 
we obtain $k$ disjoint $A$-paths. But if the number $c$ of components of $F$ is smaller than~$k$
then Lemma~\ref{l2lpaths} yields at least $\tfrac{1}{2}(|A\cap V(F)|-c)>(3k-k)/2=k$ disjoint $A$-paths.}

Thus, assume that $|A\cap V(F)|<2k+c$. Let $X$ be the union of $A\cap V(F)$ together with 
all vertices of degree~$3$ in $F$. 
In every subcubic tree, the number of vertices of degree~$3$ equals the number of leaves minus $2$. 
Thus
\[
|X|<2|A\cap V(F)|-2c<4k.
\]
We claim that $X$ meets every $A$-path in $G$. Suppose that $P$ is an $A$-path that is disjoint from $X$. 
It has to intersect $F$ since otherwise $F\cup P$ would contradict the choice of $F$. Let $P'$ be the 
initial segment of $P$ that is an $A$--$F$-path, and observe that $P'$ has length at least~$1$ since
every vertex in $A\cap V(F)$ lies in $X$, and that $P'$ does not end in a leaf of $F$ since every leaf
of $F$ lies in $X$ as well. Moreover, $P'$ does not end in a vertex of degree~$3$ either, since these also
all lie in $X$. The path $P'$ therefore ends in a vertex of degree~$2$, which means that $F\cup P$
satisfies the same conditions as $F$ and thus contradicts the choice of~$F$.
\end{proof}

\section{Long $A$-paths}\label{longsec}

For a fixed positive integer $\ell$ a cycle is \emph{long} if 
its length is at least~$\ell$. 
Long cycles have the Erd\H os-P\'osa property~\cite{BBR07,FH14,MNSW16,RS86}.
\new{Analogously}, we say a path is \emph{long} if its length is at least~$\ell$. 
Here we give a  short proof that long $A$-paths also have the Erd\H os-P\'osa property.

\begin{theorem}\label{longthm}
For every positive integer $k$, every graph $G$ and 
every set $A\subseteq V(G)$, the graph $G$ either contains $k$ disjoint long 
$A$-paths or a vertex set $X$ of size $|X|\leq 4k\ell$ that meets every long $A$-path.
\end{theorem}

For the proof we adapt, in a fairly straightforward way, a technique of Fiorini and Herinckx~\cite{FH14}
and combine it with the forest-frames of the previous section.
\begin{proof}
Let \new{$F\subseteq G$} be a forest maximal under inclusion such that
\begin{itemize}
\item $F$ \new{is subcubic with no isolated vertices}; %has minimum degree~$1$ and maximum degree at most~$3$; 
\item every leaf of $F$ lies in $A$, and every vertex in $A\cap V(F)$ is a leaf of $F$; and
\item every $A$-path in $F$ is long. 
\end{itemize}
Let $c$ be the number of components of $F$.
%We may assume that $c<k$, otherwise $G$ contains clearly $k$ disjoint long $A$-paths.
If $F$ contains $2k+c$ vertices of $A$, then Lemma~\ref{l2lpaths} applied to each component of~$F$
yields $k$ disjoint leaf-to-leaf paths in~$F$, each of which is a long $A$-path.

Thus, assume that $|A\cap V(F)|<2k+c$. 
Let $U$ be the set of vertices of degree~$3$ in~$F$. 
We extend $U$ to a vertex set $X$ as follows:
for every vertex 
in $A\cap V(F)$,
we add all vertices in $F$ at distance at most $\ell-1$ in $F$.

To bound~$|X|$ from above, assign every vertex of degree $2$ and $3$ in $F$ at distance at most $\ell-1$ to a closest leaf (\new{which is in $A$ by definition})
and in case there are multiple choices decide according to some ordering of $A$.
Observe that this results in an assignment such that all vertices that are assigned to a specific leaf 
induce a subtree of $F$.
Now, let $T$ be such a subtree, \new{and pick a root $r\in A\cap V(T)$}.
By induction on the number $m$ of vertices of degree $3$ in $T$ we prove that $|V(T)|\leq \ell + (\ell-1)\cdot m$.
If $m=0$, we have $|V(T)|\le \ell$.
Otherwise, let $p$ be a leaf in $T$ distinct from $r$. 
The $p$--$r$-path $P$ contains a vertex of degree~$3$ in $T$. 
\new{Starting from $p$,} let $u$ be the first vertex of degree 3 in $P$ and let $v$ be the predecessor of $u$
\new{in $P$}.
Deleting $pPv$ yields a subtree where $u$ has no longer degree~$3$.
Using the induction hypothesis leads to the desired result.
Hence
\begin{align*}
	|X|&\leq \ell |A\cap V(F)| + \ell |U| \\
	&=\ell(|A\cap V(F)| +|A\cap V(F)|-2c)\\
	&< 4k\ell
\end{align*}
where we have also used that the number of leaves minus $2$ equals the number of vertices of degree~$3$ in every component of $F$.

Suppose that there is a long $A$-path $P$ that is disjoint from $X$. 
By the maximal choice of $F$, the path $P$ meets $F$. 
Let $P'$ be the initial segment
of $P$ that is an $A$--$F$-path. 
Note that $P'$  ends in a vertex that has degree~$2$ in $F$,
as $(V(F)\cap A)\cup U\subseteq X$.
Finally, let $Q$ be an  $A$-path contained in $F\cup P'$ 
that starts in the $A$-vertex of $P'$ and ends in $a\in A\cap V(F)$.
Since $X$ contains all vertices of $F$ at distance at most~$\ell-1$ in $F$ to $a$, 
it follows that $Q$ is a long $A$-path.
Thus $F\cup P'$ contradicts the 
choice of $F$.
\end{proof}

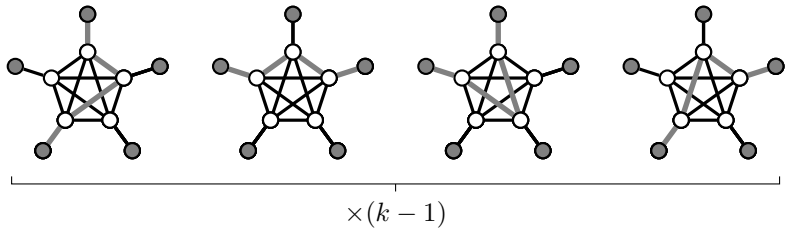
\begin{figure}[htb]
\centering
\begin{tikzpicture}
\tikzstyle{apath}=[line width=2pt,color=grau]

\def\krad{0.5}
\def\angle{72}
\def\shift{2.7}

\def\singlecomp{
\foreach \i in {0,...,4}{
  \foreach \j in {\i,...,4}{
  \begin{scope}[on background layer]
    \draw[hedge] (270-\angle/2+\i*\angle:\krad) -- (270-\angle/2+\j*\angle:\krad);
  \end{scope}
  \node[hvertex] (v\i) at (270-\angle/2+\i*\angle:\krad){};
  \node[hvertex,fill=grau] (a\i) at (270-\angle/2+\i*\angle:2*\krad){};
  \draw[hedge] (a\i) -- (v\i); 
}}
}

\singlecomp
\draw[apath] (a0) -- (v0) -- (v2) -- (v3) -- (a3);

\begin{scope}[shift={(\shift,0)}]
\singlecomp
\draw[apath] (a2) -- (v2) -- (v3) -- (v4) -- (a4);
\end{scope}

\begin{scope}[shift={(2*\shift,0)}]
\singlecomp
\draw[apath] (a4) -- (v4) -- (v1) -- (v3) -- (a3);
\end{scope}

\begin{scope}[shift={(3*\shift,0)}]
\singlecomp
\draw[apath] (a0) -- (v0) -- (v3) -- (v2) -- (a2);
\end{scope}

\draw[thin] (-2*\krad,-2.5*\krad+0.1) -- ++(0,-0.1) -- coordinate[midway] (m) (3*\shift+2*\krad,-2.5*\krad) -- ++(0,0.1); 
\draw[thin] (m) -- ++(0,-0.1) node[below]{$\times(k-1)$};

\end{tikzpicture}
\caption{Example for long $A$-paths, with vertices in $A$ in grey and $\ell=4$; 
at most $k-1$ disjoint
long $A$-paths, 
yet every hitting set needs to contain $\ell-1$ vertices from each component}\label{evenfig}
\end{figure}

How tight is the bound on the hitting set? The bound is likely not optimal, but has the right 
order of magnitude. 
To obtain a 
lower bound, take $k-1$ disjoint copies of a complete graph on $2\ell-3$ vertices whose 
vertices are matched to $2\ell-3$ vertices in $A$; see Figure~\ref{evenfig} for an 
example with $\ell=4$ and $k=5$. Then in each component of the graph, there are no two disjoint
long $A$-paths. A hitting set for long $A$-paths, on the other hand, contains at least $\ell-1$
vertices from each component. Thus, there are no $k$ disjoint long $A$-paths in the whole graph,
while the smallest hitting set has size $(k-1)(\ell-1)$.
We believe that in every graph there should be a hitting set of size at most $k\ell$. 
Our hitting set size of $4k\ell$ is only a bit off from that. 

The construction is quite similar to one proposed by Fiorini and Herinckx~\cite{FH14} for long cycles. 

\section{Even $A$-paths}\label{evensec}

The results of  Chudnovsky et al.\ or of Wollan (Theorem~\ref{wolthm}) discussed in the introduction 
imply in particular that odd $A$-paths have the Erd\H os-P\'osa property. 
What about even $A$-paths? For cycles, there is a difference between even and odd cycles. 
The former have the Erd\H os-P\'osa property, the latter do not. Interestingly, parity makes 
no difference for $A$-paths.
Again, we \new{apply} the frame argument in the proof of the theorem.

\begin{theorem}\label{eventhm}
For every positive integer $k$, every graph $G$ and 
every set $A\subseteq V(G)$, 
the graph $G$ either contains $k$ disjoint even 
$A$-paths or a vertex set $X$ of size $|X|\leq 10k$ that meets every even $A$-path.
\end{theorem}
\begin{proof}
Let \new{$F\subseteq G$} be a forest maximal under inclusion such that  
\begin{itemize}
\item $F$ \new{is subcubic with no isolated vertices}; %has minimum degree~$1$ and maximum degree at most~$3$; 
\item every leaf of $F$ lies in $A$, and every vertex in $A\cap V(F)$ is a leaf of $F$; and 
\item each component of $F$ contains an even $A$-path. 
\end{itemize}

Let $c$ be the number of components of $F$.
First assume that $|A\cap V(F)|\geq 4k+2c$. 
If $F$ has at least $k$ components then, as each component 
contains an even $A$-path, there are $k$ disjoint even $A$-paths. Thus,  $c<k$.

Consider a component $T$ of $F$. Then $T$ is a tree, and its vertices split into
two bipartition classes. The bipartition of $T$ also \new{partitions} $A\cap V(T)$;
let $A_T$ be the one  class of $A\cap V(T)$ that is not smaller than the other one. 
(If both are equal sized, pick one.) 

Each vertex in $A\cap V(T)$ is a leaf of $T$. Delete
the ones in $A\sm A_T$ and iteratively their neighbours until the 
resulting tree $T'$ has all its leaves in $A_T$ (while keeping $A_T\subseteq V(T)$). 

An application of Lemma~\ref{l2lpaths} yields 
\[
\left\lfloor\frac{|A_T|}{2}\right\rfloor\geq \left\lfloor\frac{\lceil|A\cap V(T)|/2\rceil}{2}\right\rfloor
\geq \frac{|A\cap V(T)|}{4}- \frac{1}{2}
\]
disjoint $A_T$-paths in $T$, and thus as many disjoint even $A$-paths. 

Summing over all components we  find at least 
\[
\sum_T\frac{|A\cap V(T)|}{4}- \frac{1}{2} = \frac{|A\cap V(F)|}{4}- \frac{c}{2}\geq  k
\]
disjoint even $A$-paths.

Second assume that $|A\cap V(F)|<4k+2c$. 
Let $X$ be the union of $|A\cap V(F)|$ together 
with the set of vertices of degree~$3$ in $F$. 
Since the number of leaves minus $2$ equals the number of vertices of degree~$3$ in a non-trivial subcubic tree it follows that 
\[
|X|\leq 2|A\cap V(F)|-2c \leq 8k+2c\leq 10k,
\]
where we used that $c<k$.

We claim that $X$ is a hitting set for even $A$-paths. 
Suppose that $X$ fails to meet some even $A$-path $P$. Then, $P$ cannot be 
disjoint from $F$ as otherwise $F\cup P$ would be better choice for $F$.
Thus, $P$ meets $F$\new{.} 
\new{Let} $u$ be the first vertex of $P$ in $F$ (considered from some endvertex $v$ of $P$). 
By definition of $X$, it follows that $u\notin A$ and that $u$ does not have degree~$3$ in $F$.
Therefore, $u$ has degree~$2$ in $F$, and again $F\cup vPu$ contradicts the maximal choice
of $F$.
\end{proof}

By combining the proof techniques of Theorems~\ref{longthm} and~\ref{eventhm}, 
one may readily deduce that also \emph{long even $A$-paths} have the Erd\H os-P\'osa property.
\medskip

What can we say about the  size of the hitting set?
While the bound in the theorem is not optimal, it turns out that the hitting set 
sometimes needs to be larger than $2k-2$, the optimal bound for  $A$-paths without parity constraints. 
To see this we can use the construction for long $A$-paths, where we set $\ell$ to~$4$; see 
Figure~\ref{evenfig}. 
The graph in that construction does not contain any $A$-paths of length~$2$;
that is, every even $A$-path has length at least~$4$ and is thus long. Consequently, 
the graph does not contain $k$ disjoint even $A$-paths, but
at least $3k-3$
vertices are necessary to meet every even $A$-path. 
We conjecture that a hitting set never needs more than $3k-3$ vertices.

Cycles of quite general modularity constraints have the Erd\H os-P\'osa property. 
This is the case, for instance, for cycles of length congruent to~$0$ modulo~$m$, for some fixed
positive integer~$m$ (Thomassen~\cite{Tho88}); and it is also the case for cycles of a non-zero length modulo~$m$,
whenever $m$ is odd (Wollan~\cite{Wol11}). While non-zero $A$-paths are quite well covered by the 
results of Chudnovsky et al.~and of Wollan (see Introduction), 
not much seems to be known about \emph{zero} $A$-paths,
$A$-paths of zero length modulo~$m$, or more generally, with weight $\gamma(P)=0$ for some 
directed or undirected group labeling $\gamma$ of the edges with elements of an abelian \new{group}. Is 
there a counterpart of Theorem~\ref{wolthm} for zero $A$-paths?

\begin{figure}[htb]
\centering
\begin{tikzpicture}
\tikzstyle{Apath}=[line width=2pt,grau];

\def\hstep{0.6}
\def\height{7}

\foreach \i in {1,...,\height}{
  \draw[hedge] (\hstep,\i*\hstep) -- (\height*\hstep,\i*\hstep);
  \draw[hedge] (\i*\hstep,\hstep) -- (\i*\hstep,\height*\hstep);
}

\foreach \i in {2,...,\height}{
  \draw[hedge] (\i*\hstep-\hstep,\height*\hstep) to coordinate[midway] (tl\i) (\i*\hstep,\height*\hstep);
  \path (tl\i) ++(0,0.2) node{{\small $3$}};
}

\foreach \i in {1,...,\height}{
  \foreach \j in {1,...,\height}{
    \node[hvertex] (t\i\j) at (\i*\hstep,\j*\hstep){};
  }
}

\foreach \j in {1,...,\height}{
  \node[hvertex, fill=grau] (la\j) at (0,\j*\hstep){};
  \node[hvertex, fill=grau] (ra\j) at (\height*\hstep+\hstep,\j*\hstep){};
  \draw[hedge] (la\j) to coordinate[midway] (ll\j) (t1\j);
  \draw[hedge] (ra\j) to coordinate[midway] (rl\j) (t\height\j);
  \path (ll\j) ++(0,0.2) node{{\small $1$}};
  \path (rl\j) ++(0,0.2) node{{\small $2$}};
}

\draw[Apath] (la4) -- (t14) -- (t24) -- (t25) -- (t35) -- (t45) -- (t46) -- (t47) -- (t57) -- (t56) -- (t66) -- (t76) -- (ra6);

\node at (-\hstep,0.5*\height*\hstep+0.5*\hstep) {$A$};
\node at (\height*\hstep+2*\hstep,0.5*\height*\hstep+0.5*\hstep) {$A$};

\end{tikzpicture}
\caption{All unlabeled edges have length~$6$; an $A$-path of length divisible by~$6$ in grey.}\label{zeroAfig}
\end{figure}
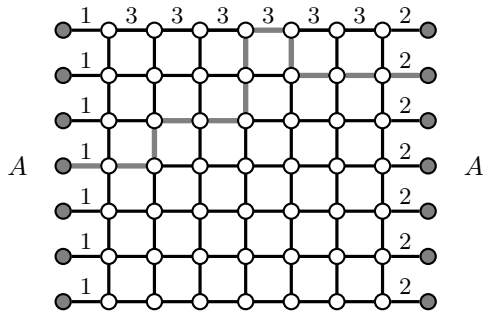

No, it turns out: already $A$-paths of length divisible by~$6$ fail to have the Erd\H os-P\'osa
property. To see this, consider the graph in Figure~\ref{zeroAfig}.
It consists of a subdivision of an $10r\times 10r$-grid, whose left side is matched to $r$ vertices in $A$,
and whose right side is linked by disjoint paths of length~$2$ to $r$ different vertices in $A$. 
Finally, the top edges of the grid are  subdivided to have length~$3$ each, while the other, unmarked edges
of the grid are subdivided to have a length of~$6$ each. 

Every $A$-path with both its endvertices on the left will have a length of $2+3s$ for some integer~$s$, while 
any $A$-path with both endvertices on the right has a length of $4+3t$, for some integer $t$. Neither of these lengths
is divisible by~$6$. Any $A$-path with a length divisible by~$6$ has to cross the grid from left to right,
and in addition, needs to pick up an odd number of the subdivided top edges. 
Clearly, no two such paths can be disjoint. On the other hand, no hitting set can have size smaller than~$r$.

We can generalise the construction to lengths divisible by $m$ for other integers than $m=6$.
Let $m>4$ be a composite number and let $p$ be its smallest prime divisor.
This implies $\frac m p>2$.
Let $b=\frac m p$ and $c=m-\frac m p - 1$.
In the construction of  Figure~\ref{zeroAfig}
we replace every length~$3$ by~$b$, every length~$2$ by~$c$ and every length of an unmarked
edge by~$m$.
Call the new graph $G(m,s)$ if the grid has size $s\times s$.
Then, any $A$-path that has both endvertices on one side 
has length congruent to $\pm 2+k\cdot \frac m p \not \equiv 0 \pmod{m}$
because $\frac m p > 2$.
Any $A$-path that crosses the grid but avoids the top edges
has a length of $m-\frac m p \not \equiv 0 \pmod{m}$.
Thus, every $A$-path of a length divisible by~$m$ crosses from left
to right and picks up at least one of the top edges. 
Again, there cannot be two such paths that are disjoint. 
Any hitting set, on the other hand, must contain a 
substantial part of the grid and thus has unbounded size. 

We can even prove a more general statement on $A$-paths with modularity constraints:
\begin{proposition}\label{dmodm}
	For any composite integer $m>4$ and every $d\in \{0,\ldots, m-1\}$, 
	the $A$-paths of length congruent to $d$ modulo $m$ do not have the Erd\H os-P\'osa property.
\end{proposition}
\begin{proof}
We start with 
 the graphs $G(m,s)$, the counterexamples for $A$-paths of length divisible by~$m$,
and then, depending on $m$ and $d$, subdivide some of the edges incident with $A$. 
An $A$-path in any subdivision of $G(m,s)$ 
is \emph{proper} if it starts on the left, intersects the top of the grid and ends on the right.

We will modify $G(m,s)$ in such a way that every proper $A$-path that intersects exactly one
subdivided edge of the top of the grid has length congruent to~$d$ modulo~$m$. That means, in 
particular, that every hitting set needs to have size proportional to~$s$.
On the other hand, we will show that no improper $A$-path can have length congruent to~$d$ modulo~$m$.
Since  no two proper $A$-paths are disjoint, this will be enough to finish the proof.

	If the equation $2x\equiv d \pmod{m}$ has a solution $x$,
	we modify the graph $G(m,s)$ as follows:
replace every edge incident with a vertex in $A$ by a path of length~$x+1$.
	Then, every $A$-path in $G(m,s)$ of length $\ell$  corresponds to an $A$-path in the modified graph
	of length $\ell+2x\equiv \ell+d$, and vice versa.
Therefore, with the same argument as above, we deduce that 
$A$-paths of length congruent to $d \pmod{m}$ in the modified graph are proper.
	
Next, assume that $2x\equiv d\pmod{m}$ does not have a solution, which implies 
that $d$ is odd but $m$ even. Thus, $p=2$.
	If $d\not\equiv \frac m 2 -2 \pmod{m}$,
	subdivide the edges in $G(m,s)$ incident with a vertex in $A$ on the left of the grid 
	such that they become paths of length $d+1$.
	Then, proper $A$-paths have length $(d+1)+r\frac m 2 + (m-\frac m 2 -1)$, which is congruent to $d\pmod m$
for some odd $r$.
Every improper $A$-path, in contrast,
has a length congruent to $2(d+1)$, $2(d+1)+\frac m 2$ (both ends on the left),
$2(\frac m 2 -1)$, $2(\frac m 2 -1)+\frac m 2$ (both ends on the right) or $(d+1)+(\frac m 2 -1)$ (no top intersection).
	Using  that $d$ is odd, $m$ is even and $d\not\equiv \frac m 2 -2 \pmod{m}$,
	we see that none of these lengths \new{are} congruent to $d$.

	Suppose now $d\equiv \frac m 2 -2$ (and $d$ is odd and $m$ is even).
	Now, subdivide the edges in $G(m,s)$ incident with a vertex in $A$ on the right side of the grid such that they become paths of length $d+1$.
%	Again, proper $A$-paths have length congruent to $d$.
	Improper $A$-paths have length congruent to $2$, $\frac{m}{2}+2$ (both ends on the left), 
	$2(\frac m 2 -1 +d)$, $2(\frac m 2 -1 +d) +\frac m 2$ (both ends on the right) 
	or $1+\frac m 2 -1+d$. 
	Using $d\equiv \frac m 2 -2$, $m>4$ and the parities of $m$ and $d$,
	we see that none of these lengths are congruent to $d \pmod{m}$.
\end{proof}

Intriguingly, the construction does not work if $m$ is a prime or equal to~$4$.
Does the 
Erd\H os-P\'osa property hold in these cases? We do not know. 

In another simple generalisation,
\new{we endow the graph with an undirected group labeling (see page \pageref{labelledGraphs}).}  
In many groups, in which there are suitable weights $b,c$ to replace the lengths~$b,c$, 
the construction can be adapted so that zero $A$-paths cannot have the Erd\H os-P\'osa 
property.

In other groups this does not seem possible.
The construction fails, for instance, when the underlying group is $\mathbb Z_2^\ell$ for some $\ell\in \N$.
This is for a reason: the proof of Theorem~\ref{eventhm} can be adapted so that it
gives the Erd\H os-P\'osa property for $\mathbb Z_2^\ell$-zero $A$-paths.

\new{
The situation seems to be more complicated for directed group labelings. 
In this setting, we currently can only construct a counterexample 
with the group~$\mathbb Z^2$ but have been unable to do so for any finite group. 
}

%%%%%%%%%%%%%%%%%%%%%%  COMMENT %%%%%%%%%%%%%%%%%%%%%%%%%%%%%%%%%%%%
% Here's a counterexample for a directed $\mathbb Z^2$-labeling:
% take the graph of Figure~\ref{zeroAfig} and put $(1,0)$ on every edge incident 
% with one of the vertices in $A$ on the left, put $(-1,0)$ on every top edge and
% put $(0,1)$ on every vertical grid edge; all other edges receive $(0,0)$.
% Interestingly, this construction does not seem to work with a finite group. 

\medskip

\new{
Arguably, the first Erd\H os-P\'osa type result is Menger's theorem. Indeed, 
weakening it somewhat, we may reformulate Menger's theorem as: $A$--$B$-paths
have the Erd\H os-P\'osa property. Strikingly, and in contrast to $A$-paths, 
the property is lost once we impose parity constraints. For instance, neither 
even nor odd $A$--$B$-paths have the Erd\H os-P\'osa property. This follows from 
an easy modification of the counterexample in Figure~\ref{zeroAfig}. We only have
to replace the right part of $A$ by $B$, and to adjust the lengths in the grid in such a
way that every $A$--$B$-paths that avoids the top edges has the wrong parity
and such that 
any $A$--$B$-path that traverses one of the top edges has the right parity.
}

\section{Combs}\label{sec:Atrees}

The forest-frame technique is not only suited for different kinds of $A$-paths
but may also be used for certain simple trees. 
One such example are \emph{combs}.

Let us define an \emph{elementary $\ell$-comb}, for an integer $\ell\geq 1$,  
as the graph obtained from a path of length~$\ell$ by adding a pendant edge to each 
internal vertex; see Figure~\ref{combfig}. An \emph{$\ell$-comb} is any subdivision of
an elementary $\ell$-comb. Finally, for a given vertex set $A$, we say that an $\ell$-comb
is an \new{\emph{$A$-$\ell$-comb}} if all its leaves are in $A$ and if every $A$-vertex in the comb 
is a leaf.

\begin{figure}[htb]
\centering
\begin{tikzpicture}

\tikzstyle{Avx}=[hvertex, fill=grau];

\def\step{0.6}

\def\comb#1{
  \draw[hedge] (0,0) -- (#1*\step+\step,0);
  \foreach \i in {1,...,#1}{
    \node[hvertex] (v\i) at (\i*\step,0){};
    \node[Avx] (a\i) at (\i*\step,-\step){};
    \draw[hedge] (v\i) -- (a\i);
  }
  \node[Avx] (l) at (0,0){};
  \node[Avx] (r) at (#1*\step+\step,0){};
}

\comb{1}
\begin{scope}[shift={(5*\step,0)}]
\comb{2}
\end{scope}
\begin{scope}[shift={(11*\step,0)}]
\comb{3}
\end{scope}

\end{tikzpicture}
\caption{An $A$-$2$-comb, an $A$-$3$-comb and an $A$-$4$-comb}\label{combfig}
\end{figure}
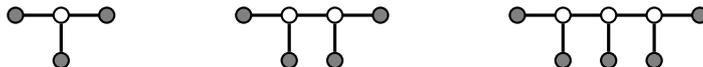

\begin{theorem}\label{combthm}
For any positive integer $\ell$, 
there exists an integer $c_\ell$ such that following holds:
For every positive integer $k$, every graph $G$, and 
every set $A\subseteq V(G)$, 
the graph $G$ either contains $k$ disjoint 
$A$-$\ell$-combs or a vertex set $X$ of size $|X|\leq c_\ell k$ that meets every $A$-$\ell$-comb.
\end{theorem}
\begin{proof}[Proof sketch]
As a frame we choose a $\subseteq$-maximal forest $F\subseteq G$ such that 
\begin{itemize}
\item $F$ \new{is subcubic with no isolated vertices}; % has minimum degree~$1$ and maximum degree at most~$3$; 
\item every leaf of $F$ lies in $A$, and every vertex in $A\cap V(F)$ is a leaf of $F$; and 
\item each component of $F$ contains an $A$-$\ell$-comb. 
\end{itemize}

Similar as with the leaf-to-leaf paths in a tree, 
we may find $\lfloor\tfrac{4}{c_\ell}|A\cap V(T)|\rfloor$ disjoint 
$A$-combs in every component $T$ of $F$, but at least one (for some positive integer $c_\ell$). 
Thus with basically the same calculations
as in the proofs of Theorems~\ref{longthm} and~\ref{eventhm}, we see that there are $k$ disjoint 
$A$-$\ell$-combs in $F$
unless the set $X$, consisting of the vertices from $A$ in $F$ and of the vertices of degree~$3$ in $F$, 
has size smaller than $c_\ell k$. 
As before we may argue that $X$ is a hitting set for $A$-$\ell$-combs.
Indeed,
\new{
suppose there is an $A$-$\ell$-comb $C$ in $G-X$.
Then,  $C$ contains an $A$--$F$-path that can be added to $F$, which results in a larger frame.
}
\end{proof}

\comment{
COMMENT START
Here's the calculation: Let $\tau$ be such that every component $T$ of $F$ yields $\lfloor \tfrac{1}{\tau}|A\cap V(T)|\rfloor$
many $A$-$\ell$-combs. Then if $|A\cap V(F)|\geq \tau k+c(\tau-1)$, where $c$ is the number of components of $F$
then we obtain $k$ disjoint combs. Also $c<k$.
Thus, assume $|A\cap V(F)|< \tau k+c(\tau-1)$. Then 
\begin{align*}
|X|\leq 2|A\cap V(F)|-2c < 2\tau k+2c\tau-4c \leq 4\tau k.
\end{align*}
Now put $c_\ell=4\tau$.
COMMENT END
}
%The argument here works for combs but not for arbitrary trees. This is because
%any tree with sufficiently many leaves contains an $\ell$-comb but not necessarily 
%a large binary tree, say. 

\section{Edge-versions and Mader's theorem}\label{edgesec}

Theorem~\ref{epthm}, the theorem of Erd\H os and P\'osa, as well as Gallai's theorem (Theorem~\ref{gallaithm}),
both have an edge-version. The one of Gallai's theorem is due to Mader.

\begin{theorem}[Mader~\cite{Mad78}]
For every positive integer $k$, every graph $G$ and 
every set $A\subseteq V(G)$, the graph $G$ either contains $k$ edge-disjoint 
$A$-paths or an edge set $F$ of size $|F|\leq 2k-2$ that meets every  $A$-path.
\end{theorem}

Mader's proof is not short. Using a frame-like argument
we give here a short proof but with a much worse bound. 
Again a tree will serve as frame:

\begin{lemma}[\new{Thomassen~\cite{Tho16}}]\label{treelem}
Let $T$ be a tree and let $A\subseteq V(T)$. Then $T$ 
contains  $\lfloor\tfrac{1}{2}|A|\rfloor$
edge-disjoint $A$-paths.
\end{lemma}

The lemma is the finite version of \new{a result of  Thomassen~\cite[Theorem~8]{Tho16}.}
Thomassen writes that the finite version is an `easy exercise'.

We now prove our weaker version of Mader's theorem.

\begin{proposition}\label{maderprop}
For every positive integer $k$, every graph $G$ and 
every set $A\subseteq V(G)$, the graph $G$ either contains $k$ edge-disjoint 
$A$-paths or an edge set $F$ of size $|F|\leq  2k\log_2k$ that meets every  $A$-path.
\end{proposition}
\begin{proof}
We may assume that $G$ is connected. Pick a spanning tree $T$ of $G$. Now, if $|A|\geq 2k$ 
then, by Lemma~\ref{treelem}, the graph $G$ contains $k$ edge-disjoint $A$-paths. 

Thus, we may assume that $|A|<2k$.
Put $\mathcal A_0=\{A\}$. Unless we find $k$ edge-disjoint $A$-paths we construct
for each $i=1,\ldots, \lceil \log_2|A|\rceil$ a partition $\mathcal A_i$ of $A$ and 
an edge set $X_i$ of size at most $ik$ such that $X_i$ meets every $B$--$B'$-path
for distinct $B,B'\in\mathcal A_i$. Assume the construction to be achieved for $i-1$. 
Split each $B\in\mathcal A_{i-1}$ into two sets $B_1$,$B_2$ that differ in cardinality by at most~$1$,
and let $\mathcal A_i$ be the set of all such $B_1,B_2$, and let $B^*_1$ be the union of all such $B_1$,
and let $B_2^*$ be the union of all such $B_2$.

\new{Apply  Menger's theorem between  $B^*_1$ and $B^*_2$.} If there are  at most $k-1$ edges
that separate $B^*_1$ from $B^*_2$ in $G-X_{i-1}$,
then we take as  $X_i$ the union of these edges with $X_{i-1}$. 
Since $X_{i-1}$ separates, by induction, $B$ from $B'$ for any two distinct $B,B'\in\mathcal A_{i-1}$,
it follows that $X_i$ separates any two sets in $\mathcal A_i$. 
If there is no edge set of size at most $k-1$ that separates $B^*_1$ from $B_2^*$, then there are
$k$ edge-disjoint $B_1^*$--$B_2^*$-paths, each of which is an $A$-path.

Note that for $s=\lceil \log_2|A|\rceil$ each set in $\mathcal A_s$ is a singleton. 
Thus, $X_s$ meets every $A$-path. Its size is $|X_s|\leq \lceil \log_2|A|\rceil \cdot k<\lceil\log_2(2k)\rceil k$.
\end{proof}

Arguably, the key argument differs markedly from the other arguments in this article, and
should perhaps not be a called a frame argument. Indeed, as a characteristic of a frame
we stated at the beginning that the frame determines the outcome: either it is large, with 
respect to some suitable measure, and then yields $k$ disjoint target objects, or it is small 
and delivers a hitting set. In the proof of Proposition~\ref{maderprop}, the frame,  the spanning tree,
only gets us halfway: if it is large, that is, if it contains many vertices from $A$, then we 
find $k$ edge-disjoint $A$-paths, but if it is small (not many $A$-vertices), then
\new{ both outcomes are still possible.} 

Why is that so? Did we not use the `right' frame? Perhaps there simply is no `right' frame.
Indeed, edge-versions in this context seem to be generally more complicated. 
To see this, let us come back to the theorem of Erd\H os and P\'osa.

Let us say that a class of graphs (or more generally objects) $\mathcal F$ has the 
\emph{edge-Erd\H os-P\'osa property} if for every integer $k$, there is a number $f(k)$ such that
every graph either contains $k$ \emph{edge-disjoint} subgraphs each isomorphic to some element of $\mathcal F$
or an \emph{edge set} $Z$ of size at most $f(k)$ that meets every subgraph contained in~$\mathcal F$.

Simonovits' proof of Theorem~\ref{epthm} can be modified so that it yields 
the edge-Erd\H os-P\'osa property for cycles, and Mader's theorem
shows that $A$-paths have it, too. Some more classes are known to have the edge-Erd\H os-P\'osa
property, but not as many as are known to have the ordinary, vertex-version, property. 
In some rare cases, the edge-property can be deduced from the vertex-property. This is, for instance,
the case for even cycles. That even cycles have the ordinary Erd\H os-P\'osa property seems to 
have been observed first by Neumann-Lara:

\begin{theorem}[Neumann-Lara, see~\cite{DNL87} or~\cite{Tho88}]\label{vxevencycthm}
% Thomassen~\cite{Tho88}
% result claimed before by Neumann-Lara, see paper of Dejter & Neumann-Lara
% what is the optimal size of the hitting set? 
% Thomassen has something exp in k
% Dejter and Neumann-Lara write: N-L used Menger and Ramsey, which certainly implies exp in k 
Even cycles have the vertex-Erd\H os-P\'osa propery. 
That is, there is a function $f$ such that 
for every positive integer $k$ every graph $G$ either contains $k$ disjoint even cycles or there is a 
vertex set $X$ of size $|X|\leq f(k)$ such that $G-X$ does not contain any even cycle. 
\end{theorem}
Chekuri and Chuzhoy~\cite{CC13} demonstrate that the size of the hitting set can be bounded 
by a function $f(k)= O(k\, \polylog k)$.

\begin{theorem}\label{thm:evenCyclesEdge}
Even cycles have the edge-Erd\H os-P\'osa property.
\end{theorem}
We use 
here a technique that is similar to the one in the proof of Theorem~\ref{eventhm}.
\begin{proof}
Let $f$ be a function as in Theorem~\ref{vxevencycthm}.
We may assume that every vertex is contained in some even cycle --- otherwise we could delete 
the vertex without changing anything.

Assume first that $G$ contains a vertex $x$ of degree at least~$6k$.
Let $c$ be the number of components of $G-x$, and consider a component $K$ of $G-x$. 
Then there is an even cycle that meets $K$.
Since the vertex set of any such even cycle is contained in $V(K)\cup\{x\}$, 
we therefore find at least $c$ edge-disjoint even cycles, one for each component of $G-x$. 
Thus, we may assume that $c\leq k-1$. 

Subdivide every edge between $x$ and $K$ exactly once, and denote the set of \new{subdividing} vertices by $A$.
In particular, $|A|= |N(x)\cap V(K)|$.  
Pick a spanning tree $T$ of $K$ (in the subdivided graph). 
The bipartition of $T$ induces a bipartition of $A\cap V(K)$.
Let $A'$ be one of the two induced bipartition classes of $A$ such that $|A'|\geq \lceil\frac{1}{2}|A\cap V(K)|\rceil$.  
Applying Lemma~\ref{treelem}, we obtain $\lfloor \tfrac{1}{2}|A'|\rfloor$ many edge-disjoint  $A'$-paths 
contained in $K$, each of which is an even $A$-path. By replacing the first and last edge of such a path $P$ 
by the edge between the second vertex of the path and $x$, and by the edge between the penultimate vertex and $x$, 
we obtain an even cycle. In this way we obtain pairwise edge-disjoint even cycles contained in $G[K+x]$.
The number of these is at least 
\[
\left\lfloor \frac{ \lceil\frac{1}{2}|N(x)\cap V(K)|\rceil}{2}\right\rfloor\geq 
\frac{1}{4}|N(x)\cap V(K)|-\frac{1}{2}.
\]
Summing over all components $K$ of $G-x$ we obtain at least 
	\begin{align*}
		\sum_{K} \left(\tfrac{1}{4}|N(x)\cap V(K)|-\tfrac{1}{2}\right) = \tfrac{1}{4}|N(x)|- \tfrac{1}{2}c
\geq k
\end{align*}
edge-disjoint even cycles, where we have used that $|N(x)|\geq 6k$ and $c\leq k$. 
	
	It remains to consider the case when no vertex in $G$ has degree at least $6k$. Since even cycles 
	have the vertex-Erd\H os-P\'osa property (Theorem~\ref{vxevencycthm}), there is a vertex set $X$ of 
	size at most $f(k)$ such that $G-X$ \new{does not contain any even cycle}. Let $F$ be the set of all edges incident 
	with any vertex in $X$. Then $F$ is an edge hitting set for even cycles of size $|F|\leq 6kf(k)$. 
\end{proof}

Normally, the edge-property does not follow as easily. 
Long cycles, for instance, do have the edge-property but in contrast to the vertex-version, 
the proof is much longer and quite a bit more complicated~\cite{BHJ17}.
While a frame argument is used,  as in our proof of Mader's theorem, the frame 
is much weaker. 
If the frame is large, then $k$ edge-disjoint long cycles are found, but if it is small, 
then more work is necessary and both outcomes \new{are  possible.} 

\comment{
Another reason why there might be no `right' frame is as follows:
let $G$ be a graph that is one of the target objects (and suppose a subdivision is also a target object) 
and let $G'$ arise from $G$ by adding a large number of parallel edges for every edge of $G$ and afterwards subdividing each edge (once).
Clearly, $G'$ contains a large number of edge-disjoint target objects, 
but for us it is unclear how a suitable frame in $G'$ should look like.
% I'd say this is not a _reason_ why there might be no right notion of a frame. Rather it is 
% evidence that there is no such thing.
}

That we know less about the edge-Erd\H os-P\'osa property becomes immediately apparent when 
we consider $A$-paths. It is an open problem, whether long $A$-paths have it.

\begin{problem}\label{longApathsEdge}
Do long $A$-paths have the edge-Erd\H os-P\'osa property?
\end{problem}

We point out that another open problem would give an affirmative answer for long $A$-paths.
\begin{problem}
Do long $A$-cycles have the edge-Erd\H os-P\'osa property?
\end{problem}

Indeed, consider a graph $G$ with a vertex set $A$, and a fixed length~$\ell$. 
Now, add a vertex $s$ and link $s$ 
to each $a\in A$ by $d(a)$ parallel paths each of length~$2$. In the resulting graph $G'$
apply the \new{edge-Erd\H os-P\'osa} property for long $\{s\}$-cycles, where we use a minimal length of $\ell+4$.
Then every long $\{s\}$-cycle contains a long $A$-path, and vice versa, every long $A$-path can be 
extended to a long $\{s\}$-cycle.
\new{
Hitting} sets may be translated in a similar fashion.

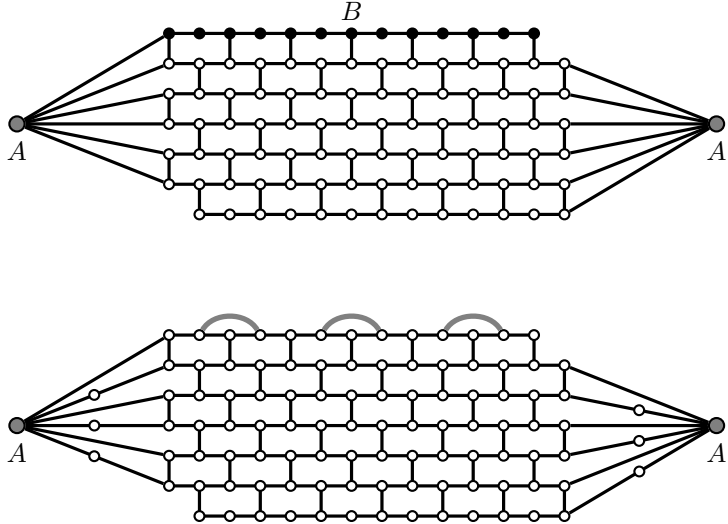
\begin{figure}[htb]
\centering
\begin{tikzpicture}

\def\wallheight{6}
\def\brickheight{0.4}

\pgfmathtruncatemacro{\lastrow}{\wallheight}
\pgfmathtruncatemacro{\penultimaterow}{\wallheight-1}
\pgfmathtruncatemacro{\lastrowshift}{mod(\wallheight,2)}
\pgfmathtruncatemacro{\lastx}{2*\wallheight+1}
\pgfmathtruncatemacro{\double}{2*\wallheight}

\tikzstyle{parityedge}=[line width=2pt,bend left=60,grau]

\node at (\wallheight*\brickheight,4+\wallheight*\brickheight+0.3){$B$};

\begin{scope}
\draw[hedge] (\brickheight,0) -- (2*\wallheight*\brickheight+\brickheight,0);
\foreach \i in {1,...,\penultimaterow}{
  \draw[hedge] (0,\i*\brickheight) -- (2*\wallheight*\brickheight+\brickheight,\i*\brickheight);
}
\draw[hedge] (\lastrowshift*\brickheight,\lastrow*\brickheight) to ++(2*\wallheight*\brickheight,0);

\foreach \j in {0,2,...,\penultimaterow}{
  \foreach \i in {0,...,\wallheight}{
    \draw[hedge] (2*\i*\brickheight+\brickheight,\j*\brickheight) to ++(0,\brickheight);
  }
}
\foreach \j in {1,3,...,\penultimaterow}{
  \foreach \i in {0,...,\wallheight}{
    \draw[hedge] (2*\i*\brickheight,\j*\brickheight) to ++(0,\brickheight);
  }
}

\def\first{0}

\foreach \i in {1,...,\lastx}{
  \node[tinyvx] (w\i\first) at (\i*\brickheight,0){};
}
\foreach \j in {1,...,\penultimaterow}{
  \foreach \i in {0,...,\lastx}{
    \node[tinyvx] (w\i\j) at (\i*\brickheight,\j*\brickheight){};
  }
}
\foreach \i in {1,...,\lastx}{
  \node[tinyvx] (w\i\lastrow) at (\i*\brickheight+\lastrowshift*\brickheight-\brickheight,\lastrow*\brickheight){};
}

\draw[parityedge] (w2\lastrow) to (w4\lastrow);
\draw[parityedge] (w6\lastrow) to (w8\lastrow);
\draw[parityedge] (w10\lastrow) to (w12\lastrow);

\node[hvertex,fill=grau,label=below:$A$] (lA) at (-2,3*\brickheight) {};
\node[hvertex,fill=grau,label=below:$A$] (rA) at (2*\wallheight*\brickheight+\brickheight+2,3*\brickheight) {};

\foreach \j in {2,4,...,\penultimaterow}{
  \draw[hedge] (lA) -- (w0\j);
}
\draw[hedge] (lA) -- (w1\wallheight);
\foreach \j in {1,3,...,\penultimaterow}{
  \draw[hedge] (lA) -- (w0\j) node[midway, tinyvx]{};
}

\foreach \j in {1,3,...,\penultimaterow}{
  \draw[hedge] (rA) -- (w\lastx\j);
}

\foreach \j in {0,2,...,\penultimaterow}{
  \draw[hedge] (rA) -- (w\lastx\j) node[midway,tinyvx]{};
}

\end{scope}

\begin{scope}[shift={(0,4)}]

\draw[hedge] (\brickheight,0) -- (2*\wallheight*\brickheight+\brickheight,0);
\foreach \i in {1,...,\penultimaterow}{
	\draw[hedge] (0,\i*\brickheight) -- (2*\wallheight*\brickheight+\brickheight,\i*\brickheight);
}
\draw[hedge] (\lastrowshift*\brickheight,\lastrow*\brickheight) to ++(2*\wallheight*\brickheight,0);

\foreach \j in {0,2,...,\penultimaterow}{
	\foreach \i in {0,...,\wallheight}{
		\draw[hedge] (2*\i*\brickheight+\brickheight,\j*\brickheight) to ++(0,\brickheight);
	}
}
\foreach \j in {1,3,...,\penultimaterow}{
	\foreach \i in {0,...,\wallheight}{
		\draw[hedge] (2*\i*\brickheight,\j*\brickheight) to ++(0,\brickheight);
	}
}

\def\first{0}

\foreach \i in {1,...,\lastx}{
	\node[tinyvx] (w\i\first) at (\i*\brickheight,0){};
}
\foreach \j in {1,...,\penultimaterow}{
	\foreach \i in {0,...,\lastx}{
		\node[tinyvx] (w\i\j) at (\i*\brickheight,\j*\brickheight){};
	}
}
\foreach \i in {1,...,\lastx}{
	\node[tinyvx, fill=black] (w\i\lastrow) at (\i*\brickheight+\lastrowshift*\brickheight-\brickheight,\lastrow*\brickheight){};
}

\node[hvertex,fill=grau,label=below:$A$] (lA) at (-2,3*\brickheight) {};
\node[hvertex,fill=grau,label=below:$A$] (rA) at (2*\wallheight*\brickheight+\brickheight+2,3*\brickheight) {};

\foreach \j in {1,...,\penultimaterow}{
	\draw[hedge] (lA) -- (w0\j);
}
\draw[hedge] (lA) -- (w1\wallheight);

\foreach \j in {0,...,\penultimaterow}{
	\draw[hedge] (rA) -- (w\lastx\j);
}
\end{scope}
\end{tikzpicture}
\caption{Counterexamples: neither $A$--$B$--$A$-paths nor even/odd $A$-paths have the edge-Erd\H os-P\'osa property}
\label{counterfig}
\end{figure}

\medskip
If some class $\mathcal F$ does not have the vertex-Erd\H os-P\'osa property, such as odd cycles, then it
usually does not have the edge-Erd\H os-P\'osa property either. This is because the counterexamples for the vertex-property
are normally based on a grid plus \new{some} extra structure; these then often (always?) turn into counterexamples
for the edge-Erd\H os-P\'osa property if a wall is used instead. (A wall is the subcubic analogue of a grid; 
for a formal definition see for instance Robertson and Seymour~\cite{RS95}).

In contrast, classes $\mathcal F$ may have the vertex-Erd\H os-P\'osa property but not the edge-Erd\H os-P\'osa property. 
We show that here  for $A$--$B$--$A$-paths,
 for even $A$-paths and for odd ones.
 To the best of our knowledge, such graph classes were not known before.

Consider the top graph in Figure~\ref{counterfig}: it consists of a wall with $10r\times 10r$ bricks together
with two vertices in $A$, one of which is linked to the left side, while the other is linked to the right side of the wall.
The top row of the wall makes up the vertex set~$B$. Clearly, no two $A$--$B$--$A$-paths can be edge-disjoint. 

At the same time, no hitting set of edges contains fewer than $r$ edges. Indeed, let $X$ be some edge set of fewer 
than $r$ edges, let $a_1$ denote the left-hand side vertex in $A$ and $a_2$ the right-hand side one.
A wall of size $10r\times 10r$ contains $10r+1$ disjoint vertical paths $P_0,\ldots, P_{10r}$,
ordered according to their starting vertex in the top row.
As $|X|<r$, there are at least two consecutive vertical paths $P_i,P_{i+1}$
such that $X$ is disjoint from $P_i\cup P_{i+1}$ and does not contain any of the two edges $e_i,e'_i$ 
with endvertices in $B$ between
the two starting vertices of $P_i$ and $P_{i+1}$ either.
Since there are, moreover, $10r$  edge-disjoint $a_1$--$P_i$-paths,
and also that many edge-disjoint $P_{i+1}$--$a_2$-paths, the set $X$ must miss 
at least one $a_1$--$P_i$-path, $Q_1$ say, and at least one $P_{i+1}$--$a_2$-path, $Q_2$ say. 
Then $Q_1\cup P_i\cup P_{i+1}\cup Q_2$ together with $e_i,e'_i$ contains an $A$--$B$--$A$-path that
avoids $X$. Thus, $X$ cannot be a hitting set of edges. This shows that $A$--$B$--$A$-paths do not 
have the edge-Erd\H os-P\'osa property.

The construction for odd (or even) $A$-paths is very similar, and shown in the bottom part of Figure~\ref{counterfig}.
Here, by subdividing certain edges incident with the left or the right $A$-vertex, 
we make sure  that every $A$-path that avoids the \new{grey} edges 
has even (resp.\ odd) length. If we define $B$ as the set of endvertices of the \new{grey} edges, then 
every odd (resp.\ even) $A$-path is an $A$--$B$--$A$-path and we may argue as above. 

\new{$A$-combs, in the sense of Section~\ref{sec:Atrees}, also fail to have the edge-Erd\H os-P\'osa property. 
The counterexamples are very similar to the ones discussed in this section.}
\comment{
\new{In a similar way we can also construct a counterexample for the edge-version of combs (the vertex version is shown in Section~\ref{sec:Atrees}): Add a new vertex to the upper graph in Figure~\ref{counterfig} and connect it with all vertices in $B$.
	Label the new vertex with $A$. 
	Then, the graph does not contain two disjoint $A$-$2$-combs
	but every edge set of small size misses at least one $A$-$2$-comb.}
}

\medskip
As in Section~\ref{evensec} we can generalise even or odd $A$-paths to zero $A$-paths with 
respect to some (directed or undirected) labeling of the edges with an abelian group~$\Gamma$. 
While in the vertex-version the group might make a difference, this is not the case
for the edge-Erd\H os-P\'osa property. Indeed, the construction shown in the bottom part
of Figure~\ref{counterfig} turns into one for zero $A$-paths: we just 
label all edges incident 
with the left vertex in $A$ with a non-zero group element $\mu\in\Gamma$, we label
all \new{grey} edges with $-\mu$ and all other edges with~$0$ (in addition, if there is a reference orientation, then orient all edges from left to right and top to bottom).

\section{Directed versions}
\label{secdirect}

Why do $A$--$B$--$A$-paths have the vertex-Erd\H os-P\'osa property but not the 
edge-version? Because in the edge-version we can force the $A$--$B$--$A$-paths 
in examples such as in Figure~\ref{counterfig} to cross the wall from left to right:
as no path can return to its starting vertex, we can fix start and endvertex in the 
edge-version by only assigning two vertices to $A$. In the vertex-version this fails,
as we could always put the two vertices in the hitting set. \new{But} if, instead, we replace the left
and the right vertex in $A$ by many vertices, then the $A$-paths can return to their starting 
side. 

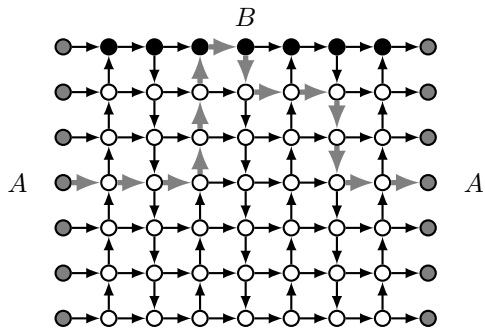
\begin{figure}[htb]
\centering
\begin{tikzpicture}
%\tikzstyle{Apath}=[line width=2pt,grau,arrows={->[scale=0.3]}];
\tikzstyle{Apath}=[line width=2pt,grau,->];
	
\def\hstep{0.6}
\def\height{7}
\def\hminusone{6}
	
\foreach \i in {1,...,\height}{
  \foreach \j in {1,...,\hminusone}{
    \node[hvertex] (t\i\j) at (\i*\hstep,\j*\hstep){};
  }
}
	
\foreach \i in {1,...,\height}{
  \node[hvertex,fill=black] (t\i\height) at (\i*\hstep,\height*\hstep){};
}
	
\node at (0.5*\height*\hstep+0.5*\hstep,\height*\hstep+0.4){$B$};
	
\foreach \i in {2,...,\height}{
  \foreach \j in {1,...,\height}{
    \pgfmathtruncatemacro{\previous}{\i-1}
    \draw[harrow] (t\previous\j) to (t\i\j);
  }
}
	
\foreach \j in {2,...,\height}{
  \foreach \i in {1,3,...,\height}{
    \pgfmathtruncatemacro{\previous}{\j-1}
    \draw[harrow] (t\i\previous) to (t\i\j);
  }
  \foreach \i in {2,4,...,\height}{
    \pgfmathtruncatemacro{\previous}{\j-1}
    \draw[harrow] (t\i\j) to (t\i\previous);
  }
}

\foreach \j in {1,...,\height}{
  \node[hvertex, fill=grau] (la\j) at (0,\j*\hstep){};
  \node[hvertex, fill=grau] (ra\j) at (\height*\hstep+\hstep,\j*\hstep){};
  \draw[harrow] (la\j) to (t1\j);
  \draw[harrow] (t\height\j) to (ra\j);
}
	
\draw[color=white, line width=4pt] (la4) -- (t14) -- (t24) -- (t34) -- (t35) -- (t36) -- (t37) -- (t47) -- (t46) -- (t56) -- (t66) -- (t65) -- (t64) -- (t74) -- (ra4);

\draw[Apath] (la4) -- (t14);
\draw[Apath] (t14) -- (t24);
\draw[Apath] (t24) -- (t34);
\draw[Apath] (t34) -- (t35);
\draw[Apath] (t35) -- (t36);
\draw[Apath] (t36) -- (t37); 
\draw[Apath] (t37) -- (t47);
\draw[Apath] (t47) -- (t46);
\draw[Apath] (t46) -- (t56);
\draw[Apath] (t56) -- (t66);
\draw[Apath] (t66) -- (t65);
\draw[Apath] (t65) -- (t64);
\draw[Apath] (t64) -- (t74);
\draw[Apath] (t74) -- (ra4);

\node at (-\hstep,0.5*\height*\hstep+0.5*\hstep) {$A$};
\node at (\height*\hstep+2*\hstep,0.5*\height*\hstep+0.5*\hstep) {$A$};
	
\end{tikzpicture}
\caption{A directed $A$--$B$--$A$-path}\label{directedfig}
\end{figure}

It is intuitively clear that we can also enforce a direction of the $A$-paths in a digraph, 
and indeed, directed $A$--$B$--$A$-paths do not have the Erd\H os-P\'osa property. 
To see this, consider the construction shown in Figure~\ref{directedfig}, where we
again scale the size of the grid to be (much) larger than the size of any purported hitting set.
Since every $A$--$B$--$A$-path needs to cross the grid from left to right, and needs to meet the top
row as well, no two disjoint such paths are possible. Any (vertex) hitting set, however, will need to 
contain a number of vertices that grows with the size of the grid. 

The construction easily transfers to the edge-Erd\H os-P\'osa property if 
the grid is replaced by a wall as in Figure~\ref{counterfig}, and, in a similar way, 
also extends to even or odd directed $A$-paths.

\bibliographystyle{amsplain}
\bibliography{../erdosposa}

\bn

\appendix
\pagebreak
\section{Overview of Erd\H os-P\'osa results}

We summarise here what is known and what not about the Erd\H os-P\'osa property
for different types of cycles and paths. 
\bigskip

\noindent
\renewcommand{\arraystretch}{1.2}
\begin{tabular}{lll}\toprule
	\bf Class & \bf Vertex property & \bf Edge property\\
\midrule
	cycles & yes \cite{EP65} & 
yes, e.g.\ \cite[Ex.~9.6]{Die10}\\
	even cycles & yes \cite{Tho88} & yes,  Theorem~\ref{thm:evenCyclesEdge}\\
	odd cycles & no \cite{DNL87} & no$^*$ \\ 
	\parbox{5cm}{cycles of length $\equiv 0 \pmod{m}$} & yes \cite{Tho88} & open\\[0.1cm]
	\parbox{5cm}{cycles of length $\not\equiv 0 \pmod{m}$\\with odd $m$} & yes \cite{Wol11} & open\\[0.2cm]
	$A$-cycles & yes \cite{PW12} & yes \cite{PW12}\\
	long cycles & yes \cite{MNSW16} & yes \cite{BHJ17}\\
	long $A$-cycles & yes \cite{BJS14} & open\\
	directed long cycles & yes \cite{KK14} & open\\
	directed $A$-cycles & no \cite{KK12} & no$^*$\\
%	{$H$-minors for planar $H$} & yes \cite{RS86} 
%	& %\parbox[t]{3.5cm}{open; Andeutung machen, dass wir schon mehr wissen?} 
%	open\\
	$A$-paths & yes \cite{Gal61} & yes \cite{Mad78}\\
	non-zero $A$-paths & yes \cite{CGGGLS06,Wol10} & no, Figure~\ref{counterfig}\\
	$A$--$B$--$A$-paths & yes \cite{KKM11} & no,   Figure~\ref{counterfig}\\
	even $A$-paths & yes, Theorem \ref{eventhm} & no, Figure~\ref{counterfig}\\[0.2cm]
	\parbox{5cm}{$A$-paths of length $\equiv d \pmod{m}$ with composite $m>4$ }& no, 
Proposition \ref{dmodm} & no, Proposition~\ref{dmodm}$^*$\\[0.4cm]
	\parbox{5cm}{$A$-paths of length $\equiv d \pmod{m}$ for $m=4$ or $m$ prime }
	& open %\parbox[t]{4cm}{open (Andeutung, dass Arthur (und Henning!) $m=4$ geloest hat?)} 
	& no, Figure~\ref{counterfig} \\[0.2cm]
	long $A$-paths & yes,  Theorem \ref{longthm} & open, Problem \ref{longApathsEdge}\\
even/odd $A$--$B$-paths & no, see Section~\ref{evensec} & no, see Section~\ref{evensec}$^*$\\
	directed $A$-paths & yes \cite{Kri05} & open\\
	directed $A$--$B$--$A$-paths & no, see Section \ref{secdirect}  & no$^*$\\
\bottomrule
\end{tabular}

\smallskip
{\small $^*$modify the counterexample by replacing the grid by a wall} 
\en

\vfill

\small
\vskip2mm plus 1fill
\noindent
Version \today{}
\bigbreak

\noindent
Henning Bruhn
{\tt <henning.bruhn@uni-ulm.de>}\\
Matthias Heinlein
{\tt <matthias.heinlein@uni-ulm.de>}\\
Institut f\"ur Optimierung und Operations Research\\
Universit\"at Ulm\\
Germany\\

\noindent
Felix Joos
{\tt <f.joos@bham.ac.uk>}\\
School of Mathematics\\ 
University of Birmingham\\
United Kingdom

\end{document}